\documentclass[12pt]{article}
\usepackage{amsmath}
\usepackage{amssymb}
\usepackage{amsthm,url,tikz}
\usepackage[utf8]{inputenc}
\usepackage[english]{babel}
\usepackage{graphicx}
\usepackage{hyperref}

\usepackage{xcolor}
\newtheorem{theorem}{Theorem}[section]

\newtheorem{lemma}[theorem]{Lemma}
\theoremstyle{definition}

\usepackage{authblk}

\usepackage[nottoc,notlot,notlof]{tocbibind}

\title{Existence of primitive pairs with two prescribed traces
over finite fields}
\author{Aakash Choudhary\footnote{ \textit{email}: achoudhary1396@gmail.com }  and R. K. Sharma \footnote{ \textit{email}: rksharmaiitd@gmail.com }}
\affil{Department of Mathematics,\\ Indian Institute of Technology Delhi-110016, India}

\date{}

\begin{document}
\maketitle

\begin{abstract}
		Given $F= \mathbb{F}_{p^{t}}$, a field with $p^t$ elements, where  $p $ is a prime power, $t\geq 7$, $n$ are positive integers and $f=f_1/f_2$ is a rational function, where $f_1, f_2$ are relatively prime, irreducible polynomials with $deg(f_1) + deg(f_2) = n $ in $F[x]$. We construct a sufficient condition on $(p,t)$ which guarantees primitive pairing $(\epsilon, f(\epsilon))$ exists in $F$ such that $Tr_{\mathbb{F}_{p^t}/\mathbb{F}_p}(\epsilon) = a$ and $Tr_{\mathbb{F}_{p^t}/\mathbb{F}_p}(f(\epsilon)) = b$ for any prescribed $a,b \in \mathbb{F}_{p}$. Further, we demonstrate for any positive integer $n$, such a pair definitely exists for large $t$. The scenario when $n = 2$ is handled separately and we verified that such a pair exists for all $(p,t)$ except from possible 71 values of $p$. A result for the case $n=3$ is given as well.
\end{abstract}


	\section{Introduction}  
	Let $\mathbb{F}_p$ represent a field of  finite order $p$, where $p=q^r$ for some prime $q$ and $r$, a positive integer. The multiplicative group  of  $\mathbb{F}_p$ is cyclic, it is denoted by $\mathbb{F}_{p}^{*}$ and a generator of $\mathbb{F}_{p}^{*}$ is referred to as a primitive element in $\mathbb{F}_{p}$. The field $\mathbb{F}_{p}$ has $\phi(p-1)$ primitive elements, where $\phi$ is the Euler's totient function. Let $\mathbb{F}_{p^{t}}$ denote an extension of $\mathbb{F}_{p}$ of degree $t$ for some positive integer $t$.
	A necessary and sufficient condition for an element $\epsilon \in \mathbb{F}^{*}_{p^t}$ to be primitive is that it is a root of an irreducible polynomial of degree $t$ over $\mathbb{F}_p$ and such an irreducible polynomial is referred to as primitive polynomial.
	For $\epsilon \in \mathbb{F}_{p^{t}}$, the trace of $\epsilon$ over $\mathbb{F}_p$ denoted by 
$Tr_{\mathbb{F}_{p^t}/\mathbb{F}_p}(\epsilon)$, is defined as $Tr_{\mathbb{F}_{p^t}/\mathbb{F}_p}(\epsilon) = \epsilon + \epsilon^p + \epsilon^{p^2}+ \dots + \epsilon^{p^{t-1}}$.  
 
In Cryptographic schemes such as  Elgamel encryption scheme and the Diffie-Hellman key exchange, primitive elements serve as the fundamental building blocks.
Numerous applications of primitive elements can be found in Coding theory and Cryptography \cite{paar}, making the study of primitive elements and primitive polynomials an active research field. Please refer to \cite{rudolf} for more information about the existence of primitive elements in finite fields.
For any rational function $f(x) \in \mathbb{F}_{p}(x)$ and $\epsilon \in \mathbb{F}_{p}$ we call the pair $(\epsilon, f(\epsilon))$, a primitive pair if both $\epsilon$ and $f(\epsilon)$ are primitive elements in $\mathbb{F}_{p}$. In general, if $\epsilon$ is primitive, $f(\epsilon)$ need not be primitive. For instance, take $x^2 + 3x +2 \in \mathbb{F}_7[x]   $, then $3,5$ are primitive elements in $\mathbb{F}_7$ but none of $f(3)$ and $f(5)$ are. 
\par
In 1985, Cohen \cite{cohen1985} introduced the term \textit{"primitive pair"} and he verified  the existence of primitive pairs $(\epsilon, f(\epsilon))$  in $\mathbb{F}_p$ for linear polynomials $f(x)= x+k \in \mathbb{F}_p[x]$. Since then   many researchers have conducted studies in this area \cite{sharma awasthi gupta, a.gupta 2017, a.gupta cohen 2018, a.gupta 2019}. Most recently, Cohen, Sharma and Sharma \cite{cohen hariom} have supplied a condition that ensures the occurrence of primitive pair
$(\epsilon, f(\epsilon))$ in $\mathbb{F}_{p}$ for non-exceptional rational function $f$, i.e., $f$ is not of the form $cx^j g^k(x)$, where $j \in \mathbb{Z}$, $k>1$ that divides $p -1$ and $c \in \mathbb{F}_{p}^{*}$, for any $g(x) \in \mathbb{F}_{p}(x)$. Jungnickel,  Vanstone  \cite{D Jungnickel} identified a sufficient condition for the occurrence of primitive elements $\epsilon \in \mathbb{F}_{p^t}$ with a prescribed trace of $\epsilon$. Later Cohen \cite{cohen 2005} extended the result with some exceptions. Chou and Cohen \cite{chou cohen}, in 2014, addressed the issue of the existence of primitive element $\epsilon \in \mathbb{F}_{p^t}$ such that
$Tr_{\mathbb{F}_{p^t}/\mathbb{F}_p}(\epsilon) = Tr_{\mathbb{F}_{p^t}/\mathbb{F}_p}(\epsilon^{-1}) = 0.$ Cao and Wang \cite{cao wang}, for $t\geq 29$, established a condition for the existence of  primitive pair $(\epsilon, f(\epsilon))$ with $f(x)=\frac{x^2 +1}{x} \in \mathbb{F}_{p^t}(x)$ such that for prescribed $a,b \in \mathbb{F}_{p}^{*}$, $Tr_{\mathbb{F}_{p^t}/\mathbb{F}_p}(\epsilon)=a$ and $Tr_{\mathbb{F}_{p^t}/\mathbb{F}_p}(\epsilon^{-1})=b$. In 2018, Gupta, Sharma and Cohen \cite{a.gupta cohen 2018}, for the same rational function and prescribed $a \in \mathbb{F}_p$, presented a condition that ensures the existence of primitive pair $(\epsilon,f(\epsilon))$ in $\mathbb{F}_{p^t}$ with $Tr_{\mathbb{F}_{p^t}/\mathbb{F}_p}(\epsilon)=a$  for $t\geq 5$. Then in 2019, Gupta and Sharma \cite{a.gupta 2019} extended the result to the rational function $\Gamma_M(x) = \frac{a_{11}x^2 + a_{12}x +a_{13}}{a_{22}x+a_{23}}$, where $M= 
\begin{pmatrix}
a_{11}&a_{12}&a_{13}\\
0&a_{22}&a_{23}
\end{pmatrix} \in \mathbb{M}_{2 \times3}(\mathbb{F}_{p^t})$ is any matrix of rank $2$, and if $\Gamma_M(x) = \lambda x$ or $\lambda x^2$ for some $\lambda \in \mathbb{F}_{p^t}$, then $\lambda =1$.  In 2021, Sharma and Sharma \cite{hariom 2} examined the rational function $f=f_1/f_2$ in $\mathbb{F}_{p^t}(x)$, where $f_1$ and $f_2$ are relatively prime, irreducible polynomials and proved that for  prescribed $a,b \in \mathbb{F}_p$, the existence of primitive pair $(\epsilon, f(\epsilon))$ in $\mathbb{F}_{p^t}$ such that
$Tr_{\mathbb{F}_{p^t}/\mathbb{F}_p}(\epsilon)=a$ and $Tr_{\mathbb{F}_{p^t}/\mathbb{F}_p}(\epsilon^{-1})=b$ for $t \geq 7$. 

\vspace{2mm}
Prior to this article, for primitive pairs, traces were considered for $\epsilon$ and $\epsilon^{-1}$. In this article, we will consider the trace onto the element $\epsilon$ and its image under $f$, i.e, $f(\epsilon)$.  
Some terminology and conventions are introduced for explanation.
We say that a non-zero polynomial $f$ over $\mathbb{F}_p[x]$ has degree $k\geq 0,$ if $f(x) = a_k x^k + a_{k-1} x^{k-1}+\dots+ a_1 x + a_0$, where $a_k\neq 0$ and we write the degree of $f$ as $deg(f)=k$. Next, we suppose that, for  a rational function $f(x) = \frac{f_1(x)}{f_2(x)} \in \mathbb{F}_{p}(x)$, $f_1$ and $f_2$ are relatively prime, irreducible polynomials and define the degree-sum as $degsum(f)= deg(f_1) + deg(f_2)$. We will now define various sets that will play a crucial role in this article.


\begin{enumerate}
    \item We define $R_{p,t}(n_1,n_2)$ to represent the set of all rational function $f(x) = \frac{f_1(x)}{f_2(x)} \in \mathbb{F}_{p^t}(x)$ such that $f_1$ and $f_2$ are relatively prime, irreducible polynomials over $\mathbb{F}_{p^t}$  with $deg(f_1)=n_1$ and $deg(f_2)=n_2$. 
    
    \item Denote  $A_{n_1,n_2}$ as the set consisting of pairs $(p,t) \in \mathbb{N}\times \mathbb{N}$ such that for any $f \in R_{p,t}(n_1,n_2)$ and prescribed $a,b \in \mathbb{F}_p$, $\mathbb{F}_{p^t}$ contains an element $\epsilon$ such that $(\epsilon, f(\epsilon))$ is a primitive pair with $Tr_{\mathbb{F}_{p^t}/\mathbb{F}_p}(\epsilon)=a$ and $Tr_{\mathbb{F}_{p^t}/\mathbb{F}_p}(f(\epsilon))=b$. 
    
    \item  Define, $R_{p,t}(n) =\bigcup_{n_1+n_2=n}R_{p,t}(n_1,n_2)$ and $A_{n}= \bigcap_{n_1+n_2=n} A_{n_1,n_2}$.
\end{enumerate}

First, in this paper, for $n \in \mathbb{N}$, we consider $f(x) \in R_{p,t}(n)$ 
and $a,b \in \mathbb{F}_p$, and then verify that there exists an element $\epsilon \in \mathbb{F}_{p^t}$ such that $(\epsilon,f(\epsilon))$ is a primitive pair in $\mathbb{F}_{p^t}$ with $Tr_{\mathbb{F}_{p^t}/\mathbb{F}_p}(\epsilon)=a$ and $Tr_{\mathbb{F}_{p^t}/\mathbb{F}_p}(f(\epsilon))=b$, i.e., we provide a sufficient condition on $p^t$ such that $(p,t) \in A_n$. Furthermore, using a sieve variation of this sufficient condition, we prove the following result:


\begin{theorem}
Let t, q, r, p $\in \mathbb{N}$ be such that q is a prime number, t $\geq 7$ and $p=q^r$. Suppose p and t assumes none of the following values:
\begin{enumerate}
    \item $2 \leq p \leq 16$ or $p=19,23,25,27,31,37,43,49,61,67,79$
    and $t=7$; 
    
    \item $2 \leq p \leq 31$ or $p=32,37,41,43,47,83$ and $t=8;$
    
    \item $2\leq p \leq 8$ or $p=11,16$ and $t=9$;
    
    \item $p=2,3,4,5,7$ and $t=10,12$;
    
    \item $p=2,3,4$ and $t=11$; 
    
    \item $p=2$ and $t=14,15,16,18,20,24$.
\end{enumerate}
Then $(p,t) \in A_2$.
\end{theorem}
\textit{Note:-} The exceptions in above theorem need not be true exceptions, they are possible exceptions.

SageMath \cite{sage} is used to perform all nontrivial calculations required throughout this article.


	\section{Preliminaries}\label{2}
	
	In this section, we present some basic concepts, notations, and results that will be used in forthcoming sections of this article.
	Throughout the article, $t$ is a positive integer, $p$ is an arbitrary prime power and $\mathbb{F}_p$ is a finite field of order $p$. 
\par 
\subsection{Definitions}
\begin{enumerate}
    \item A character of a finite abelian group $G$ is a homomorphism $\chi$ from the set $G$ into $Z^1$, where $Z^1$ is the set of all elements of complex field $\mathbb{C}$ with absolute value 1.
The trivial character of $G$ denoted by $\chi_0$, is  defined as $\chi_0(g) = 1$ for all $g \in G$.
In addition, the set of all characters of $G$, denoted by $\hat{G}$, forms
a group under multiplication, which is isomorphic to $G$. The order
of a character $\chi$ is the least positive integer $d$ such that $\chi^d= \chi_0$. For a finite field $\mathbb{F}_{p^t}$,
a character of the additive group $\mathbb{F}_{p^t}$  is called an additive character and that of the  multiplicative group $\mathbb{F}^{*}_{p^t}$
is called a multiplicative character.\\ 
\item For $u$, a divisor of $p^t -1$, an element $ \zeta \in \mathbb{F}^{*}_{p^t}$ is called $u$-$free$, if whenever $\zeta = \xi^{s}  $, where $ \xi \in \mathbb{F}_{p^t}$ and $s|u$ implies $s=1$. We see that an element $ \zeta \in \mathbb{F}^{*}_{p^t}$ is $(p^t -1)$-$free$ if and only if it is a primitive element of $\mathbb{F}_{p^t}$.
\end{enumerate}
For more information on characters, primitive elements and finite fields, we refer the reader to \cite{rudolf}.
\par
The following conclusion holds as a particular case of \cite[Lemma 10]{u free lemma}:
\begin{lemma}
Let u be a divisor of $p^t-1$, $\zeta \in \mathbb{F}_{p^t}^{*}$, then we have:
\[ 
\sum_{s|u}\frac{\mu(s)}{\phi(s)}\sum_{\chi_s}\chi_{s}(\zeta)= \left\{
\begin{array}{ll}
      \dfrac{u}{\phi(u)} & if \; \zeta \;is\; u-free, \\
      0 & otherwise \\
\end{array} 
\right. 
\]
where $\mu(.)$ is the Mobius function and $\phi(.)$ is the Euler function, $\chi_s$ runs through all the $\phi(s)$
multiplicative characters over $\mathbb{F}^{*}_{p^t}$
 with order $s$.
\end{lemma}
Therefore for $u $, a divisor of $p^t -1$ 
 \begin{align}
    \rho_u : \epsilon \mapsto \theta(u)\sum_{s|u} \frac{\mu(s)}{\phi(s)}\sum_{\chi_s}\chi_{s}(\epsilon)
 \end{align}
     gives a characteristic function for the subset of $u$-$free$ elements of $\mathbb{F}^{*}_{p^t}$, where $\theta(u) = \phi(u)/u$.
     
 Also for $a \in \mathbb{F}_p$,
 
 \begin{align}
      \tau_a : \epsilon \mapsto \frac{1}{p} \sum_{\psi \in \hat{\mathbb{F}_p}}\psi(Tr_{\mathbb{F}_{p^t}/\mathbb{F}_p}(\epsilon) -a) 
 \end{align}
 
 is a characteristic function for the subset of $\mathbb{F}_{p^t}$ whose elements satisfy  $Tr_{\mathbb{F}_{p^t}/\mathbb{F}_p}(\epsilon) =a$. From \cite[Theorem 5.7]{rudolf}, any additive character $\psi$ of $\mathbb{F}_p$ can be derived by $\psi(a) = \psi_0(ua)$,  where $\psi_0$ is the canonical additive character of $\mathbb{F}_p$ and $u$ is an element of $\mathbb{F}_p$ corresponding to $\psi$. Thus
 \begin{align}
     \tau_a &=  \frac{1}{p} \sum_{\psi \in \hat{\mathbb{F}_p}}\psi_0(Tr_{\mathbb{F}_{p^t}/\mathbb{F}_p}(u\epsilon) -ua) \nonumber
     \\
     &= \frac{1}{p} \sum_{u\in \mathbb{F}_p} \hat{\psi_0}(u\epsilon) \psi_0(-ua)
 \end{align}
 where $\hat{\psi_0}$ is the additive character of $\mathbb{F}_{p^t}$ defined by $\hat{\psi_0}(\epsilon) = \psi_0(Tr_{\mathbb{F}_{p^t}/\mathbb{F}_p}(\epsilon))$. In next theorem, we will make major use of the results given below by Wang and Fu \cite{fu wan} in 2014.
 
 \begin{lemma}{\cite[Theorem 5.5]{fu wan}}
 Let $F(x) \in \mathbb{F}_{p^d}(x)$ be a rational function. Write $F(x)= \prod_{j=1}^{k}f_{j}(x)^{r_j}$,
where $f_{j}(x) \in \mathbb{F}_{p^d} [x]$ are irreducible polynomials and $r_j$ are non zero integers. Let $\chi$ be a multiplicative character of $\mathbb{F}_{p^d}$. Suppose that the rational function $\prod_{i=1}^{d-1}f(x^{p^i})$ is not of the form $h(x)^{ord(\chi)} \in \mathbb{F}_{p^d}(x)$, where $ord(\chi)$ is the order of $\chi$, Then we have 
$$ \bigg|\sum_{\epsilon \in \mathbb{F}_p, f(\epsilon) \neq 0,\infty }  \chi(F(\epsilon)) \bigg| \leq \bigg(d\sum _{j=1}^{k}deg(f_j)-1\bigg)p^{\frac{1}{2}}.$$
 \end{lemma}

 


\begin{lemma}{\cite[Theorem 5.6]{fu wan}}
Let f(x), g(x) $\in \mathbb{F}_{p^t}(x)$ be rational functions. Write f(x) = $\prod_{j=1}^{k} f_j(x)^{r_j}$, where $f_j (x) \in \mathbb{F}_{p^t}[x]$ are irreducible polynomials and $r_j$ are non-zero integers. Let $D_1 = \sum_{j=1}^{k}deg(f_j)$, $D_2 = max\{deg(g) ,0   \}$, $D_3$ is the degree of denominator of g(x) and $D_4  $ is the sum of degrees of those irreducible polynomials dividing denominator of $g$ but distinct from $f_j(x)$( j= 1,2,...,k). Let $\chi$  be a multiplicative character of $\mathbb{F}_{p^t}$, and let $\psi$  be a nontrivial additive character of $\mathbb{F}_{p^t}$. Suppose g(x) is not of the form $v(x)^{p^t}- v(x)$ in $\mathbb{F}_{p^t}(x)$. Then we have 
$$ \bigg|\sum_{\epsilon \in \mathbb{F}_{p^t}, f(\epsilon) \neq 0,\infty, g(\epsilon) \neq \infty}  \chi(f(\epsilon)) \psi(g(\epsilon)) \bigg| \leq  (D_1 + D_2 +D_3 + D_4 -1  )p^{\frac{t}{2}}                 .$$
\end{lemma}

Evidently, both the sufficient condition (Theorem 3.1) and its sieving variation (Theorem 3.4) are entirely dependent on $p^t$ and the degrees of the numerator and denominator polynomials of the rational function. It is easy to see that the Trace part of the main result in \cite{hariom 2}  is a special case of our finding for $f(x) = \frac{1}{x}$.
 
For every $\kappa \in \mathbb{N}$, we will use $\omega(\kappa)$ to represent the number of distinct prime divisors of $\kappa$, and $W(\kappa)$ to represent the number of square free divisors of $\kappa$. Clearly, $W(\kappa) =2^{\omega(\kappa)}$.
		
\section{Sufficient Condition}\label{2}
  Let $k_1, k_2,p,t \in \mathbb{N}$ be such that $p$ is a prime power and $k_1$, $k_2$ are positive integers which divide $p^t -1.$ Let $a, b\in \mathbb{F}_p ,\; f(x) \in  R_{p,t}(n)$. 
    Let $A_{f,a,b}(k_1,k_2)$  represents the set consisting of  all those elements $\epsilon \in \mathbb{F}_{p^t}$ such that $\epsilon$ is $k_1$-free, $f(\epsilon)$ is $k_2$-free,  $Tr_{\mathbb{F}_{p^t}/\mathbb{F}_p}(\epsilon) = a$, and $Tr_{\mathbb{F}_{p^t}/\mathbb{F}_t}(f(\epsilon)) = b$.
\par
We now verify the sufficient condition as follows:

\begin{theorem}
Suppose t, n, p $\in \mathbb{N}$ and $p$ is a prime power. Suppose that 
$$ p^{\frac{t}{2}-2} > (2n+1) W(p^t -1)^2.$$
Then $(p,t) \in A_n.$
\end{theorem}

\begin{proof}
In order to prove this result it suffices to demonstrate that $A_{f,a,b}(k_1, k_2) >0$ for every $f(x) \in R_{p,t}(n)$ and for every prescribed $a,b \in \mathbb{F}_p$ .
Suppose that $f(x) \in R_{p,t}(n)$ be a rational function and that $a,b \in \mathbb{F}_p$. Let $P$ represent the collection of zeroes and poles of $f(x)  \in \mathbb{F}_{p^t}$ and $ P^{'} = P \cup \{ 0\}$.   Let $k_1, k_2$ be divisors of $p^t -1$. Then by definition, $A_{f,a,b}(k_1,k_2)$ will be given by 
    
          $$A_{f,a,b}(k_1,k_2) = \sum_{\epsilon \in \mathbb{F}_{p^t}-P^{'}} \rho_{k_1}(\epsilon) \rho_{k_2}(f(\epsilon)) \tau_a(\epsilon) \tau_b(f(\epsilon)).      $$
          Using the characteristic functions $(1)$ and $(3)$ defined in the previous section, we obtain
           $$A_{f,a,b}(k_1,k_2) =\frac{\theta(k_1) \theta(k_2)}{p^2} 
    \sum_{s_1 | k_1, s_2 | k_2} \frac{\mu(s_1) \mu(s_2)}{\phi(s_1)\phi(s_2)} \sum_{s_1,s_2} \chi_{f,a,b}(s_1,s_2)$$
    where $\theta(k_i) = \dfrac{\phi(k_i)}{k_i};\; i=1,2$ and
    $$\chi_{f,a,b}(s_1,s_2) = \sum_{u,v \in \mathbb{F}_{p}}\psi_0(-au-bv) \sum_{\epsilon \in \mathbb{F}_{p^t}-P^{'}}\chi_{s_1}(\epsilon)\chi_{s_2}(\epsilon_0) \hat{\psi_0}(u \epsilon + v \epsilon_0)$$
    where $\epsilon_0 = f(\epsilon)$. 
    It follows from \cite[Example 5.1]{rudolf} that, for any divisors $s_1 ,\; s_2$ of $p^t -1$, there exist integers $m_1, \;m_2$ with $0< m_1,\; m_2 < p^t -1$ such that $\chi_{s_1}(x) = \chi_{p^t -1}(x^{m_1})$ and $\chi_{s_2}(x) = \chi_{p^t -1}(x^{m_2})$. Thus
    \begin{align}
        \chi_{f,a,b}(s_1,s_2) &= \sum_{u,v \in \mathbb{F}_{p}}\psi_0(-au-bv) \sum_{\epsilon \in \mathbb{F}_{p^t}-P^{'}} \chi_{p^t -1}(\epsilon^{m_1} f(\epsilon)^{m_2} ) \hat{\psi_0}(u \epsilon + v \epsilon_0) 
        \\
        &=\sum_{u,v \in \mathbb{F}_{p}}\psi_0(-au-bv) \sum_{\epsilon \in \mathbb{F}_{p^t}-P^{'}} \chi_{p^t -1}(F_1(\epsilon) ) \hat{\psi_0}(F_2(\epsilon)),
    \end{align}
where $F_1(x) = x^{m_{1}}f(x)^{m_2} \in \mathbb{F}_{p^t}(x)$ and $F_2(x) = ux + vf(x) \in \mathbb{F}_{p^t}(x).$

First we consider the situation when $F_2(x) = l(x)^{p^t} - l(x)$ for some $l(x) \in \mathbb{F}_{p^t}(x)$, where $l(x)  = \dfrac{l_1(x)}{l_2(x)}$ with $(l_1, l_2) =1$. We have,
$ux+v\dfrac{f_1(x)}{f_2(x)} = \dfrac{l_1(x)^{p^t}}{l_2(x)^{p^t}} - \dfrac{l_1(x)}{l_2(x)}$, that is, 
$$f_2(x) (l_1(x)^{p^t} - l_1(x) l_2(x)^{p^t -1}) = l_2(x)^{p^t}  (u x f_{2}(x) +v f_{1}(x)   ).$$
Since $  (l_1(x)^{p^t} - l_1(x) l_2(x)^{p^t -1} , l_2(x)^{p^t}     ) = 1$, it implies that,  $l_2(x)^{p^t}$ divides $f_2(x)$, which can only happen if $l_2(x)$ is constant. That is, we have
$$c^{-(p^t)}f_2(x) (l_1(x)^{p^t} - l_1(x) c^{p^t -1}) =   u x f_{2}(x) +v f_{1}(x)$$
where $c= l_2$. Now, the above equation only applies if $v=0$. Substituting it to the equation above yields, $c^{-(p^t)} (l_1(x)^{p^t} - l_1(x) c^{p^t -1}) =   u x$, which can happen only if  $l_1$ is constant and $u=0$. Moreover, if $F_1(x) \neq r(x)^{p^t -1}$ for any $r(x) \in \mathbb{F}_{p^t}(x)$, then it follows form Lemma 2.2 that
\begin{align}
    |\chi_{f,a,b}(s_1,s_2)| \leq n p^{\frac{t}{2}+2}.
\end{align}
And, when  $F_1(x) = r(x)^{p^t -1}$ for some $r(x) \in \mathbb{F}_{p^t}(x)$,  where $r(x)  = \frac{r_1(x)}{r_2(x)}$ is such that $(r_1, r_2) =1$. Following \cite{hariom 2}, it happens only if $m_1 = m_2 = 0$, a contradiction.

If $F_2(x) \neq d(x)^{p^t} - d(x)$ for any $d(x) \in \mathbb{F}_{p^t}(x)$ then, 
\\
\textit{Case 1 :} When $n_1 \leq n_2$. Then in accordance with Lemma 2.3 we have $D_2$ = 1, and 
\begin{align}
    | \chi_{f,a,b}(s_1,s_2)| \leq (2n+1) p^{\frac{t}{2}+2}.
\end{align}
\textit{Case 2 :} When $n_1 > n_2$. We have $D_2 = n_1 - n_2$ and 
\begin{align}
    | \chi_{f,a,b}(s_1,s_2)| \leq 2n p^{\frac{t}{2} +2}.
\end{align}
Thus, if $(\chi_{s_1}, \chi_{s_2} , u, v   ) \neq (\chi_1,\chi_1,0,0)$ then  based on the discussion above, and using $(6),(7)$ and $(8)$, we get, $| \chi_{f,a,b}(s_1,s_2)| \leq (2n+1) p^{\frac{t}{2}+2}.$
From this and the definition of $A_{f,a,b}(k_1,k_2)$, we get
\begin{align}
    A_{ f,a,b}(k_1,k_2) &\geq \frac{\theta(k_1) \theta(k_2)}{p^2} ((p^t - |P^{'}|) - (2n+1)p^{\frac{t}{2}+2}(W(k_1)W(k_2)-1))
    \\
    &\geq \frac{\theta(k_1) \theta(k_2)}{p^2} ((p^t - (n+1)) - (2n+1)p^{\frac{t}{2}+2}(W(k_1)W(k_2)-1))
\end{align}
Therefore, if $p^{\frac{t}{2}-2}> (2n+1)W(k_1)W(k_2) $, then $A_{f,a,b}(k_1,k_2) >0$ for every $f(x) \in R_{p,t}(n)$ and prescribed $a,b \in \mathbb{F}_{p}$. Considering $k_1 = k_2 = p^t -1$, result follows.
\end{proof}

Now, we provide the bounds for the absolute values for $A_{f,a,b}(mk,k)-\theta(m)A_{f,a,b}(k,k)$ and $A_{f,a,b}(k,mk)-\theta(m)A_{f,a,b}(k,k).$ Proofs are omitted as they follow from the idea of \cite{a.gupta cohen 2018}.

\begin{lemma}
Let $k$ be a positive integer that divides $p^t -1$ and $m$ is a prime dividing $p^t -1$ but not $k$. Then
$$|A_{f,a,b}(mk,k)-\theta(m)A_{f,a,b}(k,k)| \leq \frac{\theta(k)^2 \theta(m)}{p^2}(2n+1)W(k)^2 p^{\frac{t}{2}+2}  $$
and
$$|A_{f,a,b}(k,mk)-\theta(m)A_{f,a,b}(k,k)| \leq \frac{\theta(k)^2 \theta(m)}{p^2}(2n+1)W(k)^2 p^{\frac{t}{2}+2}.  $$
\end{lemma}

\begin{lemma}
Let $k$ be a positive integer that divides $p^t -1$ and $\{q_1,q_2,\dots, q_m     \}$ be the collection of all primes dividing $p^t -1$ but not $k$. Then 
$$ A_{f,a,b}(p^t -1, p^t -1) \geq \sum_{i=1}^{m} A_{f,a,b}(k,q_i k) +  \sum_{i=1}^{m} A_{f,a,b}(q_ik,k) -(2m-1) A_{f,a,b,}(k,k).$$
\end{lemma}

Sieve variation of sufficient condition (Theorem 3.1) is
given below, proof of which is not given as it follows from  Lemmas 3.2, 3.3 and ideas in  \cite{a.gupta cohen 2018}.

\begin{theorem}
Let $t,n,p,k\in \mathbb{N}$ be such that $k$  divides $p^t -1 $, where $p$ is a prime power. Assume $\{q_1,q_2,\dots, q_m     \}$ is the collection of all those primes that divide $p^t -1$ but not  $k$. Suppose $\delta = 1 - 2 \sum_{i=1}^{m} \dfrac{1}{q_i}$ and $\Delta = \dfrac{2m-1}{\delta} +2$. If $\delta >0$ and 
$$ p^{\frac{t}{2}-2} > (2n+1) \Delta W(k)^2$$
then $(p,t) \in A_n.$
\end{theorem}

\begin{lemma}
Suppose that $\kappa  \in \mathbb{N}$ is such that $\omega(\kappa) \geq 1547 $, then $W(\kappa) \leq \kappa^{1/12}$.
\end{lemma}
\begin{proof}
Let $V=\{2,3,5,\dots ,12983 \}   $ is the set of first 1547 primes. We see that the product of all elements of $V$ exceeds $ K = 6.57\times 10^{5588}$. Let $
\kappa=\kappa_1 \kappa_2$, where $\kappa_1$ and $\kappa_2$ are co-prime integers such that all prime divisors of $\kappa_1$ come from the least 1547 prime divisors of $\kappa$ and remaining prime divisors
are divisors of $\kappa_2$. Hence, $\kappa_{1}^{1/12} > K^{1/12} > 5.42 \times 10^{465}$, whereas $W(\kappa_1) < 4.93 \times 10^{465}$. The conclusion follows, since $\rho^{1/12} > 2$ for all primes $\rho> 12983$.
\end{proof}
We shall need Theorem 3.4 and Lemma 3.5 for calculation work in the next section. 
 \section{Proof of Theorem 1.1}\label{2}

Proof will be carried out for the situation $t\geq 7$, since according to \cite{chou cohen}  there is no primitive element $\epsilon$, for $t\leq 4$, such that $Tr_{\mathbb{F}_{p^t}/\mathbb{F}_p}(\epsilon) = 0$ and $Tr_{\mathbb{F}_{p^t}/\mathbb{F}_p}(\epsilon^{-1}) = 0$. The cases $t = 5$ and $6$ necessitate substantial computation and appear to demand a different technique. As a result, we postpone further examination of these situations.
\par
We assume initially that, $\omega(p^t -1) \geq 1547$. Using Theorem 3.1 and Lemma 3.5, if $p^{\frac{t}{2} -2} > 5 p^{\frac{t}{6}}$, that is, if $p^t > 5^{\frac{3t}{t-6}}$ then $(p,t) \in A_2$. But $t\geq 7$ gives $\frac{3t}{t-6} \leq 21$. Hence, if $p^t > 5^{21} $ then $(p,t) \in A_2$, and this holds true for $\omega(p^t -1) \geq 1547$. Therefore, we may suppose $\omega(p^t -1) \leq 1546$. We shall use sieve variation in  order to carry forward computational work. Let $62 \leq \omega(p^t -1) \leq 1546$. To use Theorem 3.4 assume $k$ to be the product of least $62$ primes that divide $p^t -1$, that is,  $W(k) = 2^{62}$, then $m \leq 1485$ and $\delta$ assumes its least positive value when $\{ q_1,q_2,\dots,q_{1485}    \} = \{307,311,313,\dots,12979   \}$. This yields $\delta > 0.004174$ and $\Delta < 710770.7395$.  Hence $5\Delta W(k)^2 < 7.558211 \times 10^{43}$. Let $Z = 7.558211 \times 10^{43}$. By sieve variation, $(p,t) \in A_2$ if $q^{\frac{t}{2} -2} >Z$ i.e., if $p^t > Z^{\frac{2t}{t-4}}$. Since $t\geq 7$, it gives $\frac{2t}{t-4} \leq \frac{14}{3}$. Therefore, $(p,t) \in A_2$ under the condition that 
$p^t > 5.834 \times 10^{204}$. Hence, $\omega(p^t -1) \geq 95$ implies $(p,t) \in A_2$.  In a similar manner $(p,t) \in A_3 $,  $A_4$ if $\omega(p^t -1) \geq 95$,  and $(p,t) \in A_5 $ if $\omega(p^t -1) \geq 96$.
\newpage
\begin{center}
    Table 1.
\end{center}

\begin{center}
\renewcommand{\arraystretch}{1.5}
\begin{tabular}{|c| c| c| c| c| c |} 
 \hline
 $Sr.No.$& $a \leq \omega(p^t -1) \leq b$&$W(k)$& $\delta>$& $\Delta<$&$5\Delta W(k)^2<$ \\
\hline
1&$ a=13,\;b=94$ & $2^{13}$ &0.04481712&   3594.3767988&  1,206,072,718,756\\
2& $ a=7,\;b=34$  & $2^{7}$ & 0.04609692  &1151.7513186   &94,351,469 \\
3& $ a=6,\;b=25$&$2^{6}$  & 0.08241088  &450.9698124 &9,235,862   \\
4&  $ a=6,\;b=23$&  $2^{6}$  & 0.12550135  &264.9453729 & 5,426,082  \\
5& $ a=6,\;b=22$&  $2^{6}$ & 0.14959773 & 209.2223842 & 4,284,875  \\
\hline
6& $ a=5,\;b=19$  & $2^{5}$  &0.07663431  & 354.3225878 &  1,814,132  \\
\hline
7&  $ a=5,\;b=17$ & $2^{5}$ & 0.13927194  &167.1445296   & 855,780 \\
\hline
8&  $ a=5,\;b=16$ &$2^{5}$ &0.17317025  & 123.2679422  &631,132  \\
9&  $ a=5,\;b=15$ &$2^{5}$ & 0.21090610 &   92.0874844& 471,488 \\
 \hline
\end{tabular}
\end{center}

\vspace{4mm}

Using the values in the Table 1 above and repeating the process of sieve variation, we determine that $(p,t) \in A_2 $ if $p^t > (4284875)^{\frac{14}{3}}$ or $p^t > 8.8929 \times 10^{30} $ for $t\geq 7$ and since $t\geq 8$ implies $\frac{2t}{t-4}\leq 4$, so $(p,t) \in A_2 $ if 
$p^t > 3.371 \times 10^{26} $ for $t\geq 8$. 
\par
Therefore, for $t\geq 8$ it is sufficient that $\omega(p^t -1) \geq 20$, We deduce, utilising sieve variation repeatedly for values in the second section of the preceding table that, $(p,t) \in A_2$ if
$p^t > 1.084 \times 10^{25}$. 
\\
Similarly,  $\omega(p^t -1) \geq 18$ is sufficient for  inclusion of $(p,t)$ in $A_2$, and based on the table above $(p,t) \in A_2$ if $ p^t > 2.2725  \times 10^{21}$ for $t \geq 9$, and $(p,t) \in A_2$ if  $p^t > 8.158 \times 10^{18}$ for $t \geq 10$.

Hence $(p,t) \in A_2$ unless $t=7$ and $p<26382$, $ t=8$ and $p<1347$, $t=9$ and $p<237$,  $t=10$ and $p<78$, $t=11$ and $p<53$, $t=12$ and $p<38$, $t=13$ and $p<29$, $t=14$ and $p<23$, $t=15$ and $p<19$, $t=16$ and $p<16$, $t=17$ and $p<13$, $t=18$ and $p<12$, $t=19$ and $p<10$,  $t=20$ and $p<9$, $t=21,22$ and $p<8$, $t=23,24$ and $p<7$, $t=25,26,27$ and $p=2,3,4,5$. $28 \leq t \leq 31$ and $p=2,3,4$. $32 \leq t \leq 39$ and for $p=2,3$. $40 \leq t \leq 62$ and $p=2$. 

From the preceding discussion for every $(p, t)$, we validated Theorem 3.1 and compiled a list of 570 potential exceptions (listed in the Appendix). Then, for these potential exceptions, we discover that sieve variation is true for the large majority of prime powers, with the exception of those mentioned in Theorem 1.1. (see Appendix). Theorem 1.1 derives from this. Using similar reasoning, it is possible to find a subset of $A_n$ for any $n \in \mathbb{N}$. In particular, for $A_3$, we have the following result.

\begin{theorem}
Suppose t q, r, p $\in \mathbb{N}$ be such that q is a prime number, t $\geq 7$ and $p=q^r$. Let p and t assumes none of the following  values:
\begin{enumerate}
    \item $2 \leq p \leq 31$ or $p=37,41,43,49,61,67,71,79,103,121$
    and $t=7$; 
    
    \item $2 \leq p \leq 47$ or $p=53, 83$ and $t=8;$
    
    \item $2\leq p \leq 7$ or $p=9,11,16$ and $t=9$;
    
    \item $2\leq p \leq 8$ and $t=10$;
    
    \item $p=2,3,4$ and $t=11$; 
    
    \item $2\leq p \leq 7$ and $t=12$;
    
    \item $p=2$ and $t=14,15,16,18,20,24$.
\end{enumerate}
Then $(p,t) \in A_3$.
\end{theorem}

\bibliographystyle{amsplain}

\newpage

\section*{ Appendix}
\subsection*{List of 570 values of $(p,t)$ for which the condition of  Theorem 3.1 of the this article  fails:}
\subsubsection*{\textbf{For t=7:}}
$p$= 2, 4, 8, 16, 32, 64, 256, 512, 1024, 4096, 3, 9, 27, 81, 243, 729, 6561, 5, 25, 125, 625, 3125, 15625, 7, 49, 343, 2401, 11, 121, 1331, 14641, 13, 169, 2197, 17, 19, 361, 23, 529, 29, 31, 37, 41, 1681, 43, 1849, 47, 53, 59, 3481, 61, 67, 4489, 71, 79, 6241, 83, 6889, 97, 9409, 101, 103, 10609, 107, 109, 11881, 113, 127, 131, 17161, 139, 19321, 151, 22801, 157, 181, 191, 197, 199, 211, 223, 227, 229, 233, 239, 241, 263, 269, 271, 277, 281, 283, 293, 311, 331, 359, 367, 389, 397, 401, 409, 431, 439, 463, 491, 499, 509, 547, 571, 593, 601, 607, 613, 619, 631, 643, 659, 661, 683, 691, 709, 727, 733, 739, 743, 877, 919, 953, 967, 1021, 1051, 1063, 1093, 1123, 1151, 1171, 1181, 1231, 1283, 1301, 1303, 1321, 1381, 1399, 1453, 1481, 1483, 1499, 1523, 1531, 1597, 1607, 1693, 1723, 1741, 1759, 1789, 1801, 1823, 1877, 1879, 1951, 2003, 2141, 2161, 2281, 2311, 2381, 2591, 2713, 2731, 2791, 2887, 2971, 3041, 3083, 3191, 3221, 3229, 3271, 3301, 3307, 3313, 3499, 3547, 3571, 3739, 3851, 3911, 4013, 4159, 4219, 4241, 4243, 4261, 4327, 4421, 4423, 4567, 4621, 4663, 4691, 4751, 4957, 5419, 5923, 5981, 6067, 6211, 6491, 6577, 7159, 7759, 8009, 8053, 8191, 8807, 9103, 9403, 9421, 9463, 9719, 9767, 9871, 9901, 9967, 10427, 10949, 10957, 10979, 11311, 11593, 11621, 12959, 14323, 15313, 15511, 16381, 17431, 17491, 19483, 19687, 19891, 20011, 20441, 21391, 22543, 23143, 23671, 24181, 24683, 25171, 25411.

\subsubsection*{\textbf{For t=8:}}
$p$ = 2, 4, 8, 16, 32, 64, 128, 256, 512, 1024, 3, 9, 27, 81, 243, 729, 5, 25, 125, 7, 49, 343, 11, 121, 1331, 13, 169, 17, 19, 361, 23, 529, 29, 841, 31, 961, 37, 41, 43, 47, 53, 59, 61, 67, 71, 73, 79, 83, 89, 97, 101, 103, 107, 109, 113, 127, 131, 137, 139, 149, 151, 157, 163, 167, 173, 179, 181, 191, 193, 197, 211, 223, 227, 229, 233, 239, 241, 251, 257, 263, 269, 271, 277, 281, 283, 293, 307, 311, 313, 317, 331, 337, 347, 349, 353, 359, 367, 373, 379, 383, 389, 397, 401, 409, 419, 421, 433, 439, 443, 457, 461, 463, 467, 491, 499, 509, 521, 523, 541, 547, 557, 563, 569, 571, 587, 593, 599, 601, 617, 619, 631, 647, 653, 659, 661, 683, 691, 701, 709, 727, 733, 739, 743, 757, 773, 787, 797, 809, 811, 823, 827, 829, 839, 853, 857, 859, 863, 881, 887, 911, 919, 929, 937, 941, 947, 953, 967, 971, 977, 983, 991, 1009, 1013, 1021, 1033, 1039, 1061, 1063, 1069, 1087, 1091, 1093, 1097, 1103, 1109, 1117, 1123, 1201, 1213, 1217, 1223, 1231, 1277, 1279, 1283, 1289, 1291, 1301, 1303, 1319, 1321.

\subsubsection*{\textbf{For t=9:}}
$p$ = 2, 4, 8, 16, 32, 3, 9, 27, 81, 5, 25, 125, 7, 49, 11, 121, 13, 169, 17, 19, 23, 29, 31, 37, 43, 47, 53, 61, 79, 83, 137, 139, 149, 157, 211.

\subsubsection*{\textbf{For t=10:}}
$p$ = 2, 4, 8, 16, 32, 64, 3, 9, 27, 5, 25, 7, 49, 11, 13, 17, 19, 23, 29, 31, 37, 41, 53, 59, 61, 67.

\subsubsection*{\textbf{For t=11:}}
$p$ = 2, 4, 16, 3, 9, 7, 13.

\subsubsection*{\textbf{For t=12:}}
$p$ = 2, 4, 8, 16, 32, 3, 9, 27, 5, 7, 11, 13, 17, 19, 23, 29, 31, 37.

\subsubsection*{\textbf{For t=14:}}
$p$ = 2, 4, 3, 5, 7.

\subsubsection*{\textbf{For t=15:}}
$p$ = 2, 4, 16, 3, 9, 5.

\subsubsection*{\textbf{For t=16:}}
$p$ = 2, 4, 8, 3, 5.

\subsubsection*{\textbf{For t=18:}}
$p$ = 2, 3, 4, 5.

\subsubsection*{\textbf{For t=20:}}
$p$ = 2, 4, 8.

\subsubsection*{\textbf{For t=24:}}
$p$ = 2, 3, 4. 

\subsubsection*{\textbf{For t=22, 28, 30, 36:}}
$p$ = 2.

\newpage
\subsection*{List consisting values of $k$, $m$ corresponding to $(p, t)$ for
which the condition of theorem 3.1 fails but sieve variation is satisfied for
some choice of $k$ in this article.}
\subsubsection*{for t=7:}
\begin{center}
 \renewcommand{\arraystretch}{1}
\begin{tabular}{|c| c c c|} 
\hline
 $Sr.No.$ & $p$ & $k$ & $m$ \\
\hline 
1 & 32 & 1 & 4\\
2 & 64 &3 &  5\\
3 & 256 &3 & 7\\
4 &  512  &1&  6\\
5 &1024   &3&  8\\
6 & 4096 & 3 &   11\\
7 & 81 & 2  &5\\
8 & 243 &2  &4\\
9 & 729 & 2 &    7\\
10 & 6561 &2 &   8  \\ 
11 & 125 & 2  &  4\\
12 & 625 & 6   & 5\\
13 &  3125& 2  & 7\\
14 & 15625 &6   &10\\
15 & 343  &6    &4\\
16 &2401 & 6&    6\\
17 &121    & 6&   4 \\      
18& 1331&    2 &  8   \\  
19& 14641&   6  & 8     \\ 
20 &169   &  6   &4       \\
21 &  2197  & 6 &  6        \\
22 & 17      &2  & 1      \\
23 &   361   &6&   5       \\
24 & 529&     6 &  6        \\
25&  29 &    2  & 2     \\
26&  41  &   2  & 3      \\
27&   1681&  6&   6       \\ 
28 & 1849    &6  & 5        \\
29 & 47   &   2&   3      \\
30&  53      &2 & 3 \\
31&  59     &2&   5   \\
32&  3481   &6  & 8    \\
33& 4489    &6&   8    \\
\hline
\end{tabular}
\quad
\renewcommand{\arraystretch}{1}
\begin{tabular}{|c| c c c|} 
 \hline
 $Sr.No.$ & $p$ & $k$ & $m$ \\
\hline 

34& 71      &2&   4  \\
35&  6241   &6&   8 \\
36& 83      &2&   3\\
37&  6889   &6&   7\\
38& 97      &6&   3  \\
39&  9409   &6&   7  \\  
40& 101     &2&   3    \\ 
41&  103    &6&   4   \\
42& 10609   &6&   7     \\
43& 107      &2&  4  \\
44&  109    &6&   3    \\
45& 11881   &6&   7     \\
46&113     &2&    3\\
47& 127    &6&    4  \\   
48&  131    &2&   5    \\ 
49& 17161  &6&    9     \\
50& 139     &6&   4     \\
51&  19321  &6&   9     \\
52& 151     &6&   3\\
53&  22801  &6&   9  \\    
54& 157     &6&   3    \\
55& 181     &6&   4    \\
56&  191   & 2&   6    \\
57& 197     &2&   4    \\
58&  199    &6&   5    \\
59& 211     &6&   4    \\
60&  223    &6&   5    \\
61&227      &2&   4     \\
62&   229  &6&    3\\
63& 233    &2&    4\\
64&  239   &2&    5\\
65&241     &6&    3\\
66& 263    &2&    4\\
\hline
\end{tabular}
\end{center}

\begin{center}
\renewcommand{\arraystretch}{1}
\begin{tabular}{|c| c c c |} 
\hline
 $Sr.No.$ & $q$ & $l$ & $s$ \\
\hline 

67& 269     &2&   6\\
68&  271  &6& 3\\
69& 277    &6&    4\\
70& 281    &2&    5\\
71&  283   &6&    3\\
72& 293   & 2&    4\\
73& 311    &2&    5\\
74&  331   &6&    4\\
75& 359   &2& 6\\
76& 367   &6& 4\\
77& 389   &2& 6\\
78&  397  &6& 5\\
79& 401   &2& 5\\
80& 409   &6& 4\\
81&  431  &2& 7\\
82& 439   &6& 4\\
83&  463  &6& 4\\
84& 491     &2&   5\\
85 & 499    & 6&    4\\
86& 509    &2&    5\\
87& 547    &6&  4\\
88& 571    &6&  4\\
89& 593    &2&  5\\
90& 601    &6&   4\\
91& 607    &6&    4\\
92& 613    &6&    5 \\
93& 619   & 6&  4\\
94& 631    &6&    4\\
95& 643    &6&   4\\
96&  659   &2&    5\\
97& 661    &6&    5\\
98& 683    &2&    5\\
99& 691     & 6 &  6\\
100& 709    &6&  4\\
101 & 727    & 6 & 5\\
102 & 733   & 6&    4\\
103 & 739    &6&    4\\
104 & 743    &2&    5\\
105 & 877    &6&    5\\
106 & 919    &6&    6\\
\hline
\end{tabular}
\quad
\renewcommand{\arraystretch}{1}
\begin{tabular}{|c| c c c |} 
\hline
 $Sr.No.$ & $q$ & $l$ & $s$ \\
\hline

107 & 953     &2&   8\\
108 &  967   &6&    5\\
109 & 1021    &6&   5\\
110 & 1051     & 6& 5\\
111 & 1063    &6&   5\\
112 & 1093    &6&   5\\
113 & 1123    &6&  6\\
114 & 1151    &2&   6\\
115 & 1171    &6&   5\\
116 & 1181    &2&   6\\
117 & 1231    &6&  6\\
118 & 1283     &2&  6\\
119 & 1301    &2&   6\\
120 & 1303     &6&  5\\
121 & 1321    &6&   5\\
122 & 1381     &6&  5\\
123 & 1399     &6&  5\\
124 & 1453     &6&  5\\
125 & 1481    &2&   6\\
126 & 1483     &6&  6\\
127 & 1499     &2&  6\\
128 & 1523     &2&  6\\
129 & 1531     &6&  6\\
130 & 1597     &6&  5\\
131 & 1607    &2&   6\\
132 & 1693     &6&  6\\
133 &  1723   &6&   5  \\  
134 & 1741    &6&   6    \\
135 & 1759    &6&   5    \\
136 & 1789    &6&   5     \\
137 & 1801    &6&   6     \\
138 & 1823    &2&   6     \\
139 & 1877    &2&   6\\
140&  1879    &6&   5\\
141 & 1951    &6 &  7  \\ 
142 & 2003    &2&   7    \\ 
143 & 2141    &2&   7     \\
144 & 2161    &6&   7     \\
145 & 2281    &6&   6     \\
146 & 2311    &6&   6     \\
\hline
\end{tabular}

\end{center}

\begin{center}
\renewcommand{\arraystretch}{1}
\begin{tabular}{|c| c c c |} 
\hline
 $Sr.No.$ & $p$ & $k$ & $m$ \\
\hline 

147 & 2381    &2&   8     \\
148 & 2591    &2&   7     \\
149 & 2713    &6&   6     \\
150 & 2731    &6&   7     \\
151 & 2791    &6&   7     \\
152 & 2887    &6&   6     \\
153 & 2971    &6&   6     \\
154 & 3041    &2&   7     \\
155 & 3083    &2&   7     \\
156 & 3191    &2&   7     \\
157 & 3221    &2&   7     \\
158 & 3229    &6&   6    \\
159 & 3271    &6&   6     \\
160 & 3301    &6&   7     \\
161 & 3307    &6&   6     \\
162 & 3313    &6&   6     \\
163 & 3499    &6&   8    \\
164 & 3547    &6&   8    \\
165 & 3571    &6&   6     \\
166 & 3739    &6&   6     \\
167 & 3851    &2&   7     \\
168 & 3911    &2&   7     \\
169 & 4013    &27 &  7     \\
170 & 4159    &6&   6     \\
171 & 4219    &6&   7     \\
172 & 4241    &2&   7     \\
173 & 4243    &6&   7     \\
174 & 4261    &6&   6 \\
175 & 4327    &6&   7   \\  
176 & 4421    &2&   7     \\
177 & 4423    &6&   6     \\
178 & 4567    &6&   6     \\
179 & 4621    &6&   6     \\
180 & 4663    &6&   6     \\
181 & 4691    &2&   7     \\
182 & 4751    &2&   7     \\
183 & 4957    &6&   7     \\
184 & 5419    &6&   7     \\
185 & 5923    &6&   7     \\
186 & 5981    &2&   8    \\
187 & 6067    &6&   7     \\
188 & 6211    &6&   7     \\
189 & 6491    &2&   8 \\
\hline
\end{tabular}
\quad
\renewcommand{\arraystretch}{1}
\begin{tabular}{|c| c  c c |} 
\hline
 $Sr.No.$ & $p$ & $k$ & $m$ \\
\hline 

190 & 6577    &6&   7   \\  
191 & 7159    &6&   7    \\
192 & 7759    &6&   8    \\
193 & 8009    &2&   9     \\
194 & 8053    &6&   9     \\
195 & 8191    &6&   7     \\
196 & 8807    &2&   8     \\
197 & 9103    &6&   7     \\
198 & 9403    &6&   7     \\
199 & 9421    &6&   7     \\
200 &  9463   &6&   7    \\
201 & 9719    &2&   8     \\
202 & 9767    &2&   8     \\
203 & 9871    &6&   7     \\
204 & 9901    &6&   7     \\
205 & 9967    &6&   7     \\
206 & 10427   &2&   8     \\
207 & 10949   &2&   10     \\
208 & 10957   &6&   8     \\
209 & 10979   &2&   8     \\
210 & 11311   &6&   7     \\
211 & 11593   &6&   7     \\
212 & 11621   &2&   8     \\
213 & 12959   &2&   9     \\
214 & 14232   &6&   8 \\
215 & 15313   &6&   8   \\  
216 & 15511   &6&   8     \\
217 & 16381   &6&   8 \\
218 & 17431   &6&   9   \\  
219 & 17491   &6&   9    \\
220 & 19483   &6&   8    \\
221 & 19687   &6&   8     \\
222 & 19891   &6&   8     \\
223 & 20011   &6&   8     \\
224 & 20441   &2&   10     \\
225 & 21391   &6&   8     \\
226 & 22543   &6&   8     \\
227 & 23143   &6&   8     \\
228 & 23671   &6&   9     \\
229 & 24181   &6&   8     \\
230 & 24683   &2&   9     \\
231 & 25171   &6&   9 \\
232 & 25411   &6&   8   \\
\hline
\end{tabular}
\end{center}

\newpage
\subsubsection*{for t=8:}
\begin{center}
 \renewcommand{\arraystretch}{1}
\begin{tabular}{|c| c c c|} 
 \hline
 $Sr.No.$ & $p$ & $k$ & $m$ \\
\hline 
1 & 64 &6 &  7\\
2 &128&3&6\\
3 & 256 &3 & 6\\
4 &  512  &6&  10\\
5 &1024   &3&  8\\
6 &81   & 2&   5 \\
7 &243 &2 &6\\
8 &125& 6& 6\\ 
9 &49& 6& 4\\
10 &343& 6& 9\\
11 &121& 6& 5\\
12 &169& 6& 5\\
13 &361& 6& 6\\
14 &529 &6 & 7 \\
15 &841& 6& 7 \\
16 &961 &6 &7 \\
17 &53& 6& 5 \\
18 &59 &6 &6\\
19 &61& 6& 4 \\
20 &67 &6 &6\\
21 &71& 6& 4 \\
22 &73& 6& 5\\
23 &79 &6 &6\\
24 &89& 6& 6\\
25 &97 &6 &5\\
26 &101& 6 &5\\
27 &103 &6 &5 \\
28 &107& 6& 5 \\
29 &109 &6 &6 \\
30 &113& 6& 5 \\
31 &127 &6 &6 \\
32 &131& 6& 6 \\
33 &137 &6 &7 \\
34 &139& 6& 6 \\
35& 149 &6 &6 \\
36 &151& 6& 6 \\
37 &157 &6 &7 \\
38 &163& 6& 5 \\

\hline
\end{tabular}
\quad
\renewcommand{\arraystretch}{1}
\begin{tabular}{|c| c c c|} 
 \hline
 $Sr.No.$ & $p$ & $k$ & $m$ \\
\hline 
39 &167 &6 &6 \\
40 &173& 6& 7 \\
41 &179& 6& 6 \\
42 &181& 6& 6 \\
43 &191 &6 &7\\
44 &197& 6& 6 \\
45 &211 &6 &7 \\
46 &223& 6& 7 \\
47 &227 &6 &7 \\
48 &229& 6& 8 \\
49 &233 &6 &8 \\
50 &239& 6& 7 \\
51 &241 &6 &7 \\
52 &251& 6& 5 \\
53 &257& 6& 5 \\
54 &263 &6 &7 \\
55 &269& 6& 6 \\
56 &271 &6 &6 \\
57 &277& 6& 6 \\
58 &283 &6 &6 \\
59 &293& 6& 7 \\
60 &311 &6 &7 \\
61 &313& 6& 6 \\
62 &317 &6 &6\\
63 &331& 6& 8 \\
64 &337 &6 &6 \\
65 &347& 6& 6 \\
66 &349 &6 &6 \\
67 &353& 6& 6 \\
68 &359 &6 &7 \\
69 &367& 6& 6 \\
70 &373 &6& 7 \\
71 &379& 6& 6 \\
72 &383& 6& 7\\
73 &389& 6& 8 \\
74 &397 &6 &6 \\
75 &401& 6& 8 \\
76 &409 &6 &6 \\

\hline
\end{tabular}
\end{center}

\begin{center}
\renewcommand{\arraystretch}{1}
\begin{tabular}{|c| c c c|} 
 \hline
 $Sr.No.$ & $p$ & $k$ & $m$ \\
\hline 
77 &433& 6& 7\\
78 &439 &6& 7\\
79 &443& 6& 7\\
80 &457 &6 &7 \\
81 &467& 6& 7\\
82 &491 &6 &7 \\
83 &499& 6& 8\\
84 &509 &6 &7\\
85 &521& 6& 7 \\
86 &523 &6 &6\\
87 &541& 6& 6 \\
88 &547 &6 &8\\
89 &557& 6& 7\\
90 &563 &6 &9 \\
91 &569& 6& 6\\
92 &571 &6 &9\\
93 &587& 6& 7\\
94 &593 &6 &8 \\
95 &599& 6& 8 \\
96 &601 &6 &7 \\
97 &617& 6& 8 \\
98 &619 &6 &7\\
99 &631& 6& 7 \\
100 &647 &6 &7 \\
101 &653& 6& 7\\
102 &661 &6 &7 \\
103 &683& 6& 7\\
104 &691 &6 &7\\
105 &701& 6& 7 \\
106 &709 &6 &7 \\
107 &733& 6& 7 \\
108 &739 &6 &7 \\
109 &743& 6& 9 \\
110 &757& 6& 8 \\
111 &773 &6 &8\\
112 &787& 6& 8 \\
113 &797 &6 &8 \\
114 &809& 6& 8 \\

\hline
\end{tabular}
\quad
\renewcommand{\arraystretch}{1}
\begin{tabular}{|c| c c c|} 
 \hline
 $Sr.No.$ & $p$ & $k$ & $m$ \\
\hline 
115 &811 &6 &7 \\
116 &823& 6& 8\\
117 &827 &6 &8\\
118 &829& 6& 7\\
119 &839 &6 &7\\
120 &857& 6& 7 \\
121 &859 &6 &8 \\
122 &863& 6& 8\\
123 &881 &6 &7\\
124 &887& 6& 7 \\
125 &919 &76 &7 \\
126 &929 &6 &7 \\
127 &941& 6& 7 \\
128 &947 &6 &8 \\
129 &953& 6& 7 \\
130 &971 &6 &8\\
131 &977& 6& 8\\
132 &983 &6 &7\\
133 &991& 6& 7\\
134 &1009 &6 &7\\
135 &1013& 6& 7 \\
136 &1021 &6 &8\\
137 &1033& 6& 8\\
138 &1039 &6 &8\\
139 &1061&  6& 9 \\
140 &1063 &6 &8\\
141 &1069& 6& 7\\
142 &1087 &6 &7\\
143 &1091& 6& 8\\
144 &1093& 6 &7 \\
145 &1097& 6& 8\\
146 &1103 &6 &8\\
147 &1109 &6 &8 \\
148 &1117 &6 &8 \\
149 &1123 &6 &8\\
150 &1201 &6 &9 \\
151 &1213 &6 &8 \\
152 &1223 &6 &9 \\

\hline
\end{tabular}
\end{center}

\begin{center}
\renewcommand{\arraystretch}{1}
\begin{tabular}{|c| c c c|} 
\hline
 $Sr.No.$ & $p$ & $k$ & $m$ \\
\hline 
153 &1231 &6 &8 \\
154 &1277 &6 &9 \\
155 &1279 &6 &8 \\
156 &1283 &6 &9 \\
157 &1289 &6 &9  \\
158 &1291 &6 &9 \\
159 &1303 &6 &8 \\
160 &1319 &6 &8 \\
161 &1321 &6 &8\\
162 &191 &6 &7 \\
163 &911 &30 &9\\
164 &659 &30 &8 \\
165 &1301 &30 &10\\

\hline
\end{tabular}
\quad
\renewcommand{\arraystretch}{1}
\begin{tabular}{|c| c c c |} 
\hline
 $Sr.No.$ & $p$ & $k$ & $m$ \\
\hline 
166 &281 &30 &6 \\
167 &410 &30 &6 \\
168 &421 &30 &8 \\
169 &937 &30 &7 \\
170 &1331 &30 &9\\
171 &307 &30 &7 \\
172 &1217 &30 &7 \\
173 &967 &30 &8 \\
174 &461 &30 &7 \\
175 &463 &30 &7\\
176 &853 &30 &9 \\
177 &727 &30 &7 \\
178 &729 &30 &9\\

\hline
\end{tabular}
\end{center}

\subsubsection*{for t=9:}
\begin{center}
\renewcommand{\arraystretch}{1}
\begin{tabular}{|c| c c c|}
\hline
 $Sr.No.$ & $p$ & $k$ & $m$ \\
 \hline
1 &8&7&2\\
2 &32&7&5\\
3 &27&2&5\\
4&81&14&7\\
5 &25&6&5\\
6 &125&2&7\\
7&49&6&5\\
8&121&6&6\\
9&13&6&2\\
\hline 
\end{tabular}
\renewcommand{\arraystretch}{1}
\begin{tabular}{|c| c c c |}
\hline
 $Sr.No.$ & $p$ & $k$ & $m$ \\
\hline 
10&169&6&7\\
11&17&2&3\\
12&19&6&3\\
13&23&2&5\\
14&29&2&5\\
15&31&6&4\\
16&37&6&5\\
17&43&6&6\\
18&47&2&5\\
\hline
\end{tabular}
\renewcommand{\arraystretch}{1}
\begin{tabular}{|c| c c c |} 
\hline
 $Sr.No.$ & $p$ & $k$ & $m$ \\
\hline 
19&53&2&7\\
20&61&6&5\\
21&79&6&5\\
22&83&2&6\\
23&137&2&7\\
24&139&6&6\\
25&149&2&7\\
26&157&6&6\\
27&211&6&7\\
\hline
\end{tabular}
\end{center}

\subsubsection*{for t=10: }
\begin{center}
\renewcommand{\arraystretch}{1}
\begin{tabular}{|c| c c c|}
\hline
 $Sr.No.$ & $p$ & $k$ & $m$ \\
 \hline
1 &8&3&5\\
2&16&3&6\\
3&32&3&6\\
4&64&6&9\\
5 &9&2&4\\
6&27&2&7\\
7&25&6&6\\
\hline 
\end{tabular}
\quad
\renewcommand{\arraystretch}{1}
\begin{tabular}{|c| c c c |} 
\hline
 $Sr.No.$ & $p$ & $k$ & $m$ \\
\hline 
8&49&6&6\\
9 &11&6&3\\
10&13&6&4\\
11&17&6&4\\\
12&19&6&5\\
 13&23&6&5\\
14&29&6&6\\
\hline
\end{tabular}
\quad
\renewcommand{\arraystretch}{1}
\begin{tabular}{|c| c c c |} 
\hline
 $Sr.No.$ & $p$ & $k$ & $m$ \\
\hline 
15&31&6&5\\
16&37&6&5\\
17 &41&6&6\\
18&53&6&6\\
19&59&6&7\\
20&61&6&6\\
21 &67&6&6\\
\hline
\end{tabular}
\end{center}

\subsubsection*{for t=11:}

\begin{center}
\renewcommand{\arraystretch}{1}
\begin{tabular}{|c| c c c|}
\hline
 $Sr.No.$ & $p$ & $k$ & $m$ \\
 \hline
1&16&3&6\\
2&9&2&4\\
3&7&2&3\\
4&13&2&5\\
\hline 
\end{tabular}
\end{center}

\subsubsection*{for t=12:}
\begin{center}
\renewcommand{\arraystretch}{1}
\begin{tabular}{|c| c c c|}
\hline
 $Sr.No.$ & $p$ & $k$ & $m$ \\
 \hline
1 &8&39&6\\
2&16&39&7\\
3&32&39&9\\
4&9&2&6\\
5&13&6&6\\
\hline 
\end{tabular}
\quad
\renewcommand{\arraystretch}{1}
\begin{tabular}{|c| c c c |} 
\hline
 $Sr.No.$ & $p$ & $k$ & $m$ \\
\hline 
6&17&6&6\\
7&19&6&6\\
8&29&6&9\\
9&27&30&6\\
& & &\\
\hline
\end{tabular}
\quad
\renewcommand{\arraystretch}{1}
\begin{tabular}{|c| c c c |} 
\hline
 $Sr.No.$ & $p$ & $k$ & $m$ \\
\hline 
10&11&30&6\\
11&23&30&7\\
12&31&30&6\\
13&37&30&8\\
& & &\\
\hline
\end{tabular}
\end{center}

\subsubsection*{for t=14:}
\begin{center}
\renewcommand{\arraystretch}{1}
\begin{tabular}{|c| c c c|}
\hline
 $Sr.No.$ & $p$ & $k$ & $m$ \\
 \hline
1&4 &15& 4\\
2& 5 &6& 3\\
3& 7& 6& 4\\
4&3& 2& 2\\
\hline 
\end{tabular}
\end{center}

\subsubsection*{for t=15:}
\begin{center}
\renewcommand{\arraystretch}{1}
\begin{tabular}{|c| c c c|}
\hline
 $Sr.No.$ & $p$ & $k$ & $m$ \\
 \hline
1 &4 &3& 5\\
2 &3& 2& 3\\
3 &9& 2& 7\\
4 &5& 2& 5\\
5 &16& 15& 9\\
\hline 
\end{tabular}
\end{center}

\subsubsection*{for t=16:}
\begin{center}
\renewcommand{\arraystretch}{1}
\begin{tabular}{|c| c c c|}
\hline
 $Sr.No.$ & $p$ & $k$ & $m$ \\
 \hline
1 &4 &3& 4\\
2 &8& 3& 8\\
3 &3 &2 &4\\
4 &5& 2& 5\\
\hline 
\end{tabular}
\end{center}

\subsubsection*{for t=18:}
\begin{center}
\renewcommand{\arraystretch}{1}
\begin{tabular}{|c| c c c|}
\hline
 $Sr.No.$ & $p$ & $k$ & $m$ \\
 \hline
1 &3& 2& 5\\
2 &4 &15 &6\\
3 &5 &6 &5\\
\hline 
\end{tabular}
\end{center}

\subsubsection*{for t=20:}
\begin{center}
\renewcommand{\arraystretch}{1}
\begin{tabular}{|c| c c c|}
\hline
 $Sr.No.$ & $p$ & $k$ & $m$ \\
 \hline
1& 4& 3& 6\\
2 &8 &15& 9\\

\hline 
\end{tabular}
\end{center}	

\subsubsection*{for t=22:}
\begin{center}
\renewcommand{\arraystretch}{1}
\begin{tabular}{|c| c c c|}
\hline
 $Sr.No.$ & $p$ & $k$ & $m$ \\
 \hline
1 &2& 15& 2\\
\hline 
\end{tabular}
\end{center}

\subsubsection*{for t=24:}
\begin{center}
\renewcommand{\arraystretch}{1}
\begin{tabular}{|c| c c c|}
\hline
 $Sr.No.$ & $p$ & $k$ & $m$ \\
 \hline
1 &4& 3& 8\\
2 &3& 2& 6\\
\hline 
\end{tabular}
\end{center}

\subsubsection*{for t=28:}
\begin{center}
\renewcommand{\arraystretch}{1}
\begin{tabular}{|c| c c c|}
\hline
 $Sr.No.$ & $p$ & $k$ & $m$ \\
 \hline
1 &2& 15& 4\\

\hline 
\end{tabular}
\end{center}

\subsubsection*{for t=30:}
\begin{center}
\renewcommand{\arraystretch}{1}
\begin{tabular}{|c| c c c|}
\hline
 $Sr.No.$ & $p$ & $k$ & $m$ \\
 \hline
1 &2 &3 &5\\
\hline 
\end{tabular}
\end{center}

\subsubsection*{for t=36:}
\begin{center}
\renewcommand{\arraystretch}{1}
\begin{tabular}{|c| c c c|}
\hline
 $Sr.No.$ & $p$ & $k$ & $m$ \\
 \hline
1 &2 &15 &6\\

\hline 
\end{tabular}
\end{center}

\end{document}